\documentclass{amsart}

\usepackage{color}
\usepackage{hyperref}

\title[Symmetric Kronecker products and wave packets]{Symmetric Kronecker products and\\ semiclassical wave packets}
\author[George A. Hagedorn]{George A. Hagedorn}
\address[George A. Hagedorn]{Department of Mathematics and
Center for Statistical Mechanics, Mathematical Physics, and
Theoretical Chemistry, 
Virginia Polytechnic Institute and State University, 
Blacksburg, Virginia 24061-0123, U.S.A}
\email{hagedorn@math.vt.edu}

\author[Caroline Lasser]{Caroline Lasser}
\address[Caroline Lasser]{Zentrum Mathematik M3, Technische Universit\"at M\"unchen,
D-80290 M\"unchen, Germany}
\email{classer@ma.tum.de}

\date{\today}


\def\C{{\mathbb C}}
\def\N{{\mathbb N}}
\def\R{{\mathbb R}}

\def\Id{{\rm Id}}

\def\Rr{{\mathcal R}}

\newtheorem{theorem}{Theorem}
\newtheorem{proposition}{Proposition}
\newtheorem{lemma}{Lemma}
\newtheorem{corollary}{Corollary}
\newtheorem{definition}{Definition}
\newtheorem{remark}{Remark}

\begin{document}
%\tableofcontents

\begin{abstract}
We investigate the iterated Kronecker product of a square matrix
with itself and prove an invariance property for symmetric subspaces. 
This motivates the definition of an iterated symmetric Kronecker product
as a restriction of the iterated 
product on a symmetric subspace 
and the derivation of an explicit formula for its action on vectors. 
We apply our result for describing a linear change in the matrix 
parametrization of semiclassical  wave packets. 
\end{abstract}

\maketitle

\section{Introduction}

The Kronecker product of matrices is known to be ubiquitous \cite{VL00},
and our aim here is to investigate the $n$-fold Kronecker product of 
a complex square matrix $M\in\C^{d\times d}$ with itself, 
\[
M^{n\otimes} \ = \
\underbrace{M \otimes \cdots \otimes M}_{n\,\text{times}},\qquad n\in\N,
\]
and to apply our findings to the parametrization of semiclassical wave packets. 

\subsection{The motivation}
We encountered a variant of the $n$-fold Kronecker product when
studying linear changes in the para\-metrization of semiclassical wave packets. 
Semiclassical wave packets have first been proposed in \cite{Hag85} as a 
multivariate non-isotropic generalization of the Hermite functions. 
See also \cite{Hag98}.
A family of semiclassical wave packets 
\[
\left\{\varphi_{\boldsymbol{k}}[A,\,B; \boldsymbol{a},\,\boldsymbol{\boldsymbol{\eta}}]: \boldsymbol{k}\in\N^d\right\}
\]
is parametrized by two invertible complex matrices $A,\,B\in{\rm GL}(d,\C)$ and 
two real vectors $\boldsymbol{a},\boldsymbol{\boldsymbol{\eta}}\in\R^d$. It 
 forms an orthonormal basis of the Hilbert space of square integrable 
functions. Here, we focus on the more delicate dependence 
on the parametrizing matrices $A$ and $B$. We therefore take $\boldsymbol{a}=\boldsymbol{\eta}=\boldsymbol{0}$ and simply write 
$\varphi_{\boldsymbol{k}}[A,\,B]$ for the corresponding wave packet. 
A wave packet with $|\boldsymbol{k}|=n$ is the product 
of a multivariate polynomial of order $n$ times a complex-valued Gaussian.

If the parameter matrix $A$ has real entries only, 
then the polynomial can be factorized into 
univariate Hermite polynomials. A linear change of the 
parametrization, 
\[
A'=AM, \quad B' = BM\quad\text{for some }M\in{\rm GL}(d,\C),
\]
results in a formula for the wave packet $\varphi_{\boldsymbol{k}}[A',\,B']$ 
involving wave packets in the old parametrization weighted by coefficients 
stemming from the $n$-fold Kronecker product $M^{n\otimes}$. 
The following analysis will reveal the relevant symmetric subspaces 
and corresponding orthogonal projections such that the resulting 
$n$-fold symmetric Kronecker product explicitly describes the wanted change 
of the parametrization.

\subsection{Two-fold symmetric Kronecker products}
In semidefinite programming
(See for example \cite{AHO98} or \cite[Appendix~E]{Kle02}.),
the two-fold Kronecker product has notably occurred in combination with 
subspaces of a particular symmetry property. One considers the space
\[
X_2 \ =\ \left\{\boldsymbol{x}\in\C^{d^2}: \boldsymbol{x} = {\rm vec}(X),
\,X\ =\ X^t\in\C^{d\times d}\right\},
\]
that contains those vectors that can be obtained by the row-wise vectorization
of a complex symmetric $d\times d$ matrix, that is, a complex matrix coinciding with its transpose matrix.
The dimension of the space $X_2$ is
\[
L_2 \ =\  \tfrac12\,d\,(d+1).
\]
One can prove that this space is invariant under Kronecker products,
in the sense that for all matrices $M\in\C^{d\times d}$, one has
\[
(M\otimes M)\,\boldsymbol{x}\in X_2,\quad \text{whenever}\;\; \boldsymbol{x}\in X_2.
\]
Now one uses the standard basis of $\C^{d^2}$ for constructing an
orthonormal basis of the subspace $X_2$ and
defines a corresponding sparse $L_2\times d^2$ matrix~$P_2$ that has the
basis vectors as its rows. The symmetric Kronecker product of $M$ with itself
is then the $L_2\times L_2$ matrix
\[
S_2(M) \ =\ P_2\,\left(M\otimes M\right)\, P_2^*.
\]

\subsection{$\bf n$-fold symmetric Kronecker products}
How does one extend this construction to symmetrizing $n$-fold Kronecker 
products? It is instructive to revisit the second order space in two dimensions 
and to write a vector $\boldsymbol{x}\in X_2$ as
\[
\boldsymbol{x}\ =\ (x_{(2,0)},\,x_{(1,1)},\,x_{(1,1)},\,x_{(0,2)})^t.
\]
This labelling uses the multi-indices $\boldsymbol{k}=(k_1,\,k_2)\in\N^2$
with $k_1+k_2=2$ in the redundant enumeration
\[
\boldsymbol{\nu_2} \ =\  \left((2,0),\,(1,1),\,(1,1),\,(0,2)\right).
\]
This description allows for a straightforward extension
to higher order $n$ and dimension $d$.  
One works with a redundant enumeration of the multi-indices
$\boldsymbol{k}=(k_1,\,\ldots,\,k_d)\in\N^d$ with $k_1+\cdots+k_d=n$, 
collects them in a row vector $\boldsymbol{\nu_n}$, 
and defines
\[
X_n \ = \ \left\{ \boldsymbol{x}\in\C^{d^n}: \ \text{For all } j,j'\in\{1,\ldots,d^n\}, \ x_{j} = x_{j'}\;\;\text{if}\;\;
\boldsymbol{\nu_n}(j)=\boldsymbol{\nu_n}(j')\right\}.
\]
The dimension of $X_n$ equals the number of mult-indices in $\N^d$
of order $n$, that is the binomial coefficient
\[
L_n\ =\ \binom{n+d-1}{n}.
\]
And again, we can prove invariance in the sense that for all
$M\in\C^{d\times d}$
\[
M^{n\otimes}\,\boldsymbol{x}\in X_n,\quad\text{whenever}\;\; \boldsymbol{x}\in X_n.
\]
See Proposition~\ref{prop:kron}. Then, we use the standard basis of
$\C^{d^n}$ to build an orthonormal basis of $X_n$ and assemble the
corresponding sparse $L_n\times d^n$ matrix $P_n$.
All this motivates the definition of the $n$-fold
symmetric Kronecker product as
\[
S_n(M) \ =\ P_n\,M^{n\otimes}\,P_n^*.
\]
The matrix $S_n(M)$ is of size $L_n\times L_n$ and inherits structural
properties as invertibility or unitarity from the matrix $M$.
See Lemma~\ref{lem:str}.

\bigskip
Our main result Theorem~\ref{theo:main} provides an explicit formula for the 
action of the matrix $S_n(M)$ in terms of 
multinomial coeffients and powers of the entries of the original matrix~$M$. 
Labelling the components of a vector $\boldsymbol{y}\in\C^{L_n}$
by multi-indices of order $n$, 
we obtain for all $\boldsymbol{k}\in\N^d$ with $|\boldsymbol{k}|=n$ that
\begin{align*}
&\left(S_n(M)\,\boldsymbol{y}\right)_{\boldsymbol{k}}
=\,\frac{1}{\sqrt{\boldsymbol{k}!}}\sum_{|\boldsymbol{\boldsymbol{\alpha_1}}|=k_1}\cdots
\sum_{|\boldsymbol{\boldsymbol{\alpha_d}}|=k_d}\,\binom{k_1}{\boldsymbol{\boldsymbol{\alpha_1}}}\cdots
\binom{k_d}{\boldsymbol{\boldsymbol{\alpha_d}}}\,\boldsymbol{\boldsymbol{m_1}}^{\boldsymbol{\boldsymbol{\alpha_1}}}\cdots \boldsymbol{\boldsymbol{m_d}}^{\boldsymbol{\boldsymbol{\alpha_d}}}\\ 
&\hspace*{20em} 
\times\ \sqrt{(\boldsymbol{\boldsymbol{\alpha_1}+\cdots+\boldsymbol{\alpha_d}})!}
\ y_{\boldsymbol{\boldsymbol{\alpha_1}}+\cdots+\boldsymbol{\boldsymbol{\alpha_d}}},
\end{align*}
where $\boldsymbol{\boldsymbol{m_1}},\,\ldots,\,\boldsymbol{\boldsymbol{m_d}}\in\C^d$ denote the row vectors of $M$. 
The summations range over multi-indices $\boldsymbol{\boldsymbol{\alpha_1}},\ldots,\boldsymbol{\boldsymbol{\alpha_d}}\in\N^d$ with 
$|\boldsymbol{\boldsymbol{\alpha_1}}|=k_1$, \ldots, $|\boldsymbol{\boldsymbol{\alpha_d}}|=k_d$. They are weighted with multinomial coefficients stemming from the $n$-fold Kronecker product 
$M^{n\otimes}$, whereas the square roots of the factorials originate 
in the orthonormalization of the row vectors of the matrix~$P_n$.

\subsection{Application to semiclassical wave packets}
We consider semiclassical wave packets $\varphi_{\boldsymbol{k}}[A,\,B]$, $\boldsymbol{k}\in\N^d$, parametrized by two
matrices $A,\,B\in{\rm GL}(d,\,\C)$. The $\boldsymbol{k}^{\mbox{\scriptsize th}}$ wave packet is the  
product of a multivariate polynomial $p_{\boldsymbol{k}}[A]$ and the complex-valued Gaussian function 
\[
\varphi_{\boldsymbol{0}}[A,B](x) \ =\  (\pi\,\hbar)^{-d/4}\,\det(A)^{-1/2}\,
\exp\!\left(\,-\,\frac{\langle \boldsymbol{x},\,B\,A^{-1}\,\boldsymbol{x}\rangle}{2\,\hbar}\right),\qquad \boldsymbol{x}\in\R^d,
\]
 where $\hbar>0$ is the semiclassical parameter used for the overall scaling. See Definition~\ref{def:packet}. 
 The polynomial family 
 \[
 \left\{p_{\boldsymbol{k}}[A]: \boldsymbol{k}\in\N^d\right\}
 \] 
 obeys a three-term recurrence relation and a Rodrigues type representation. See \cite[Proposition~4]{LT14} and  \cite[Theorem~4.1]{Hag15}. It is orthogonal with respect to the Gaussian weight function 
 $|\,\varphi_{\boldsymbol{0}}[A,B](x)\,|^2$, but differs from the standard Hermite polynomials on $\R^d$ that are biorthogonal and not orthogonal. See for example \cite[\S6]{IZ17}. 
 If the matrix $A$ is real, then the polynomials $p_{\boldsymbol{k}}[A]$ are real and factorize into univariate scaled 
 Hermite polynomials. In the complex case, we encounter a more intricate structure that we wish to explore both for theoretical and numerical reasons.
 
 \bigskip
 We consider a change of parametrization
\[
A'=A\,M, \qquad B'=B\,M
\]
induced by a suitably chosen invertible matrix $M\in{\rm GL}(d,\C)$.
Collecting all wave packets $\varphi_{\boldsymbol{k}}[A,\,B]$ of order  $|\boldsymbol{k}|=n$ as the components of a 
formal vector of wave functions $\vec\varphi_n[A,\,B]$, the  formula of 
Theorem~\ref{theo:main} allows us to identify the change of parametrization explicitly as
\[
\vec\varphi_n[A',\,B']\ =\ \det(M)^{-1/2}\,S_n(M)\,
\vec\varphi_n[A,\,B],
\]
see Corollary~\ref{MainResult} in Section \ref{Section5.3}. That is, the $n$-fold symmetric 
Kronecker product explicitly transforms one parametrization into another one.

\bigskip
Recently, E.~Faou, V.~Gradinaru, and C.~Lubich \cite{FGL09,L} have used semiclassical wave packets for the numerical  
discretization of semiclassical quantum dynamics. See also \cite{GH14}. 
The computationally demanding step of this method is the assembly of the Galerkin matrix 
for the potential function $V:\R^d\to\R$ according to
\[
\left\langle \varphi_{\boldsymbol{k}}[A,\,B], V\varphi_{\boldsymbol{l}}[A,\,B]\right\rangle\ =\
\int_{\R^d}\,\overline{\varphi_{\boldsymbol{k}}[A,\,B](x)}\, V(x)\,\varphi_{\boldsymbol{l}}[A,\,B](x)\,dx, 
\]
where the multi-indices $\boldsymbol{k},\boldsymbol{l}\in\N^d$ are bounded in modulus by some truncation value~$N\in\N$, 
that determines the dimension of the Galerkin space. 
If the wave packets are parametrized by a matrix $A$ that has only real entries, then they 
factorize into univariate Hermite functions, and the multi-dimensional integral becomes the product 
of one-dimensional ones. \cite[Chapter 5.9]{B17} presents a two-dimensional numerical test case, 
transforming a linear combination of semiclassical wave packets of order $n=4$ from one parametrization to another one  
using a tree-based implementation of the $n$-fold symmetric Kronecker product. The 
transformation error is in the order of machine precision. This experiment suggests 
a new numerical method for semiclassical quantum dynamics using the change 
of parametrization via  $n$-fold symmetric Kronecker products. Such a method assembles
the Galerkin matrix in terms of univariate Hermite functions. 
Then, the known large order asymptotics of the Hermite functions should allow one to stabilize the 
numerical evaluation of the integrands \cite{TTO16}, such that larger values of the truncation value $N$ become feasible.

\subsection{Organization of the paper}
In the next Section,%~\ref{sec:com},
we start with some combinatorics for 
explicitly relating the lexicographic enumeration of multi-indices of order $n$ 
with the redundant enumeration $\boldsymbol{\nu_n}$. Then we introduce the symmetric 
subspaces $X_n$ in Section~\ref{sec:sym} and construct an orthonormal 
basis together with the 
corresponding matrix $P_n$. There we also discuss symmetric subspaces 
and our basis construction in tensor terminology.
In Section \ref{sec:kron}, we define the
$n$-fold symmetric Kronecker product and prove our main results 
Proposition~\ref{prop:kron} and Theorem~\ref{theo:main}.
An introduction to semiclassical wave packets and the description of linear 
changes in their parametrization by symmetric Kronecker products is given in 
Section~\ref{sec:sem}.

\subsection{Notation}
Vectors and multi-indices are bold.
On some occasions we shall use the binomial coefficient
\[
\binom{n}{j}\ =\ \frac{n!}{(n-j)!\,j!},\qquad
\text{for non-negative integers }n\ge j.
\]
We write a multi-index $\boldsymbol{k}\in\N^d$ as a row vector $\boldsymbol{k}=(k_1,\,\ldots,\,k_d)$.
We use the modulus $|\boldsymbol{k}|=k_1+\cdots+k_d$, and the multinomial coefficient
\[
\binom{|\boldsymbol{k}|}{\boldsymbol{k}}\ =\ \frac{|\boldsymbol{k}|!}{k_1!\,\cdots\,k_d!},\qquad
\text{for}\;\;\boldsymbol{k}\in\N^d.
\]
We adopt the convention that any multinomial coefficient with any negative 
argument is defined to be $0$. We also use the $\boldsymbol{k}^{\mbox{\scriptsize th}}$ power of a vector, 
\[
\boldsymbol{x}^k\ =\ x_1^{k_1}\,\cdots\,x_d^{k_d},\qquad \boldsymbol{x}\in\C^d.
\]

\section{Combinatorics}\label{sec:com}

\subsection{Reverse Lexicographic ordering} 
First we enumerate the set multi-indices of order $n$ in $d$ dimensions,
\[
\left\{\boldsymbol{k}\in\N^d: |\boldsymbol{k}|=n\right\},\qquad n\in\N,
\]
in reverse lexicographic ordering and collect them
as components of a formal row vector denoted by $\boldsymbol{\ell_n}$.
The length of the vector $\boldsymbol{\ell_n}$ is the binomial coefficient
\[
L_n\ =
%\ \frac{(n+d-1)!}{n!\ (d-1)!}\ =
\ \binom{n+d-1}{n}
%=\\binom{n+d-1}{d-1}.
\]
%Indeed, the length $L_n$ of the vector $\boldsymbol{\ell_n}$ is the number of ways of 
%obtaining the integer~$n$ as a summation of $d$ non-negative integers, 
%which is $(d+n-1)$ choose $n$. See {\it e.g.}, \cite[Chapter 1, \S7]{Ber71}.
%Alternatively, 
One can think of this in the following way \cite{JHT}:
The multi-indices $\boldsymbol{k}$ of order $n$ in $d$ dimensions are in a one-to-one
correspondence with the sequences of $n$ identical balls and $d-1$
identical sticks. The sticks partition the line into $d$ bins into which one can
insert the $n$ balls.
(The first bin is to the left of all the sticks, and contains $k_1$
balls; the last bin is to the right of all the sticks, and contains $k_d$ balls; for
$2\le j\le d-1$, the $j^{\mbox{\scriptsize th}}$ bin
is between sticks $j-1$ and $j$, and it contains $k_j$ balls.)
{\it E.g.}, the multi-index $(3,\,2,\,0,\,1)$ in four dimensions
corresponds to
$$
{\color{red}\bullet}\quad{\color{red}\bullet}\quad{\color{red}\bullet}\quad
{\color{blue}|}\quad{\color{red}\bullet}\quad{\color{red}\bullet}\quad
{\color{blue}|}\quad{\color{blue}|}\quad{\color{red}\bullet}.
$$
If all these objects were distinguishable, there would be $(n+d-1)!$ 
permutations, but since the balls are all identical, one must divide by $n!$,
and since the sticks are all identical, one must divide by $(d-1)!$.

\subsection{A redundant enumeration}
Next we redundantly enumerate and collect multi-indices of modulus $n$ in a 
 vector $\boldsymbol{\nu_n}$ of length $d^n$. Each entry of the vector~$\boldsymbol{\nu_n}$ is a multi-index of modulus $n$. Some of these entries occur repeatedly, since our enumeration is redudant. We proceed recursively and set
\[
\boldsymbol{\nu_0}=\left((0,\,\ldots,\,0)\right),\quad
\boldsymbol{\nu_{\boldsymbol{1}}}=(\boldsymbol{\boldsymbol{e_1}}^t,\,\ldots,\,\boldsymbol{e_d}^t),\quad
\]
and
\[
\boldsymbol{\nu_{n+1}}\ =\ {\rm vec}
\begin{pmatrix}\boldsymbol{\nu_n}(1)+\boldsymbol{\boldsymbol{e_1}}^t & \ldots & \boldsymbol{\nu_n}(d^n)+\boldsymbol{\boldsymbol{e_1}}^t\\ 
\vdots & & \vdots\\ \boldsymbol{\nu_n}(1)+\boldsymbol{e_d}^t& \ldots & 
\boldsymbol{\nu_n}(d^n)+\boldsymbol{e_d}^t\end{pmatrix},\qquad n\ge0,
\]
where $\boldsymbol{\boldsymbol{e_1}},\ldots,\boldsymbol{e_d}\in\C^d$ are the standard basis vectors of $\C^d$,
and ${\rm vec}$ denotes the row-wise vectorization of a matrix into a row vector.

\vskip 5mm
For example, for $d=2$, we have
\begin{align*}
\boldsymbol{\ell_1} \ &= \ \left( (1,0),\,(0,1)\right),\\
 \boldsymbol{\nu_1} \ &= \ \left( (1,0),\,(0,1)\right),\\
 \boldsymbol{\ell_2}\ &=\ \left((2,0),\,(1,1),\,(0,2)\right),\\
\boldsymbol{\nu_2}\ &=\ \left((2,0),\,(1,1),\,(1,1),\,(0,2)\right),\\
\boldsymbol{\ell_3}\ &=\ \left((3,0),\,(2,1),(1,2),\,(0,3)\right),\\
\boldsymbol{\nu_3}\ &=\ \left((3,0),\,(2,1),\,(2,1),\,(1,2),\,(2,1),\,(1,2),\,(1,2),\,
(0,3)\right).
\end{align*}
We observe that the multi-index $(1,1)$ appears twice in $\boldsymbol{\nu_2}$, since 
\[
(1,1) = \boldsymbol{\nu_{\boldsymbol{1}}}(1)+\boldsymbol{e_2}^t = \boldsymbol{\nu_{\boldsymbol{1}}}(2)+\boldsymbol{\boldsymbol{e_1}}^t.
\]
The modulus three multi-index $(2,1)$ can be generated as 
\[
(2,1) = \boldsymbol{\nu_2}(1)+\boldsymbol{e_2}^t = \boldsymbol{\nu_2}(2)+\boldsymbol{\boldsymbol{e_1}}^t = \boldsymbol{\nu_2}(3)+\boldsymbol{\boldsymbol{e_1}}^t,
\] 
and therefore appears three times in $\boldsymbol{\nu_3}$.

\subsection{A partition}
For relating the lexicographic and the redundant enumeration,
we define the mapping
\[
\sigma_n: \{1,\,\ldots,\,L_n\}\to {\mathcal P}(\{1,\,\ldots,\,d^n\}) 
\]
so that for all $i\in\{1,\,\ldots,\,L_n\}$ and $j\in\{1,\,\ldots,\,d^n\}$
the following holds:
\[
j\in\sigma_n(i)\ \Longleftrightarrow\ \boldsymbol{\nu_n}(j)=\boldsymbol{\ell_n}(i).
\]

\vskip 5mm
For example, for $d=2$, we have
\[
\sigma_2(1)=\{1\},\quad \sigma_2(2)=\{2,\,3\},\quad
\sigma_2(3)=\{4\}
\]
and
\[
\sigma_3(1)=\{1\},\quad \sigma_3(2)=\{2,\,3,\,5\},
\quad\sigma_3(3)=\{4,\,6,\,7\},\quad\sigma_3(4)=\{8\}.
\]

\vskip 5mm
We observe the following partition property.

\begin{lemma}\label{lem:sigma}
We have 
\[
\#\sigma_n(i)\ =\ \binom{n}{\boldsymbol{\ell_n}(i)},\qquad i=1,\,\ldots,\,L_n,
\] 
and
$\displaystyle\bigcup_{i=1,\,\ldots,\,L_n}\sigma_n(i) = \{1,\,\ldots\,d^n\}$,
~where the union is pairwise disjoint.
\end{lemma}

\begin{proof}
We first prove that we have a partition property. For any
$j\in\{1,\,\ldots,\,d^n\}$ there exists $i\in\{1,\,\ldots,\,L_n\}$ so that 
$\boldsymbol{\nu_n}(j)=\boldsymbol{\ell_n}(i)$. So, we clearly have
\[
\bigcup_{i=1,\ldots,L_n} \sigma_n(i) = \{1,\,\ldots,\,d^n\}.
\]
Moreover, since $j\in\sigma_n(i)\cap\sigma_n(i')$ is equivalent to 
$\boldsymbol{\ell_n}(i)=\boldsymbol{\ell_n}(i')$, that is, $i=i'$, the union is disjoint. 

For proving the claimed cardinality,  we argue by induction.
For $n=0$, we have
\[
\boldsymbol{\ell_{0}}\ =\ \left((0,\,\ldots,\,0)\right)\ =\ \boldsymbol{\nu_{0}},\qquad
\sigma_{0}(1) = \{1^0\},\qquad
\#\sigma_{0}(1) = 1.
\]

For the inductive step,
we observe that in the redundant enumeration $\boldsymbol{\nu_n}$, 
the multi-index $\boldsymbol{k}=(k_1,\,\ldots,\,k_d)$ can be generated from $d$ possible
entries in $\boldsymbol{\nu_{n-1}}$,  
\[
(k_1-1,\,\,k_2,\,\ldots,\,k_d),\ \ldots,\ (k_1,\,\ldots,\,k_{d-1},\,k_d-1)
\] 
by adding $\boldsymbol{\boldsymbol{e_1}}^t,\,\ldots,\,\boldsymbol{e_d}^t$, respectively.
Of course, such an entry only  belongs to $\boldsymbol{\nu_{n-1}}$
if all its components are non-negative.
For each of these indices with all entries non-negative, there is a unique
number $j\in\{1,\,2,\,\dots,\,L_{n-1}\}$, such that $\boldsymbol{\ell_{n-1}}(j)$
is the given index.
If one of these indices has a negative entry, we define $j=-1$ and
$\sigma_{n-1}(-1)$ to be the empty set, {\it i.e.},
\[
\sigma_{n-1}(-1)\ =\ \{\},\quad\mbox{whose cardinality is } 0.
\]
We list the $d$ numbers defined this way as $i_1,\,\ldots,\,i_d$, and note
that all the positive values in this list must be distinct.
Then,
\begin{eqnarray}\nonumber
\#\sigma_n(i)&=& 
\sum_{m=1}^d\,\#\sigma_{n-1}(i_m)
\\[3mm]\nonumber
&=&
\sum_{m=1}^d\,\binom{n-1}{k_1,\,\ldots,\,k_m-1,\,\ldots,\,k_d}
\\[3mm]\nonumber
&=&
\sum_{m=1}^d\,\left.
\begin{cases}\frac{(n-1)!}{k_1!\,\cdots(k_m-1)!\,\cdots\,k_d!}&k_m>0\\ 
0&k_m=0\end{cases}
\right\}
\\[3mm]\nonumber
&=& \frac{(n-1)!\,(k_1+\cdots+k_d)}{k_1!\,\cdots\,k_d!}
\\[3mm]\nonumber
&=&\binom{n}{\boldsymbol{k}}.
\end{eqnarray}
\end{proof}

\begin{remark}
Consider $i\in\{1,\ldots,L_n\}$. A number $j\in\{1,\ldots,d^n\}$ is contained in the set $\sigma_n(i)$, if the multi-index $\boldsymbol{\nu_n}(j)$ coincides with the multi-index $\boldsymbol{\ell_n}(i)$. The previous Lemma~\ref{lem:sigma} verifies that the cardinality $\#\sigma_n(i)$ is the number of unique permutations of the multi-index $\boldsymbol{\ell_n}(i)$.
\end{remark}

\section{Symmetric subspaces}\label{sec:sym}
We next analyze the symmetric spaces
\[
X_n \ = \ \left\{ \boldsymbol{x}\in\C^{d^n}: \ \text{For all } j,j'\in\{1,\ldots,d^n\}, \ x_{j} = x_{j'}\;\;\text{if}\;\;
\boldsymbol{\nu_n}(j)=\boldsymbol{\nu_n}(j')\right\}
\]
for $n\in\N$.
%\[
%X_n\ =\ \left\{\boldsymbol{x}\in\C^{d^n}: x_j = x_{j'}\;\;\text{if}\;\; 
%\boldsymbol{\nu_n}(j)=\boldsymbol{\nu_n}(j')\right\},\qquad n\in\N.
%\] 
We have $X_0=\C$ and $X_1=\C^d$, whereas $X_n$ is a proper subset of 
$\C^{d^n}$ for $n\ge2$.

\vskip 5mm
For example, for $d=2$,
\begin{align*}
X_2\ &=\ \left\{\boldsymbol{x}\in\C^4: x_2=x_3\right\},\\
X_3\ &=\ \left\{\boldsymbol{x}\in\C^8: x_2=x_3=x_5,\ x_4=x_6=x_7\right\}. 
\end{align*}

\vskip 5mm
Any vector $\boldsymbol{x}\in X_n$ has $d^n$ components, but the components that
correspond to the same multi-index in the redundant enumeration
$\boldsymbol{\nu_n}(1),\,\ldots,\,\boldsymbol{\nu_n}(d^n)$ have the same value. 
Hence, at most $L_n$ components of $\boldsymbol{x}\in X_n$ are  different.
They may be labelled by the multi-indices $\boldsymbol{\ell_n}(1),\,\ldots,\,\boldsymbol{\ell_n}(L_n)$,
and we often refer to them by
\[
x_{\boldsymbol{\ell_n}(i)},\qquad i=1,\,\ldots,\,L_n.
\]

\vskip 5mm
The symmetric subspaces $X_n$, $n\in\N$, can also be obtained as 
the vectorization of symmetric tensor spaces, 
and we next relate this alternative point of view to ours.

\subsection{The symmetric spaces in tensor terminology}
The second order subspace
\[
X_2\ =\ \left\{\boldsymbol{x}\in\C^{d^2}: x_j = x_{j'}\;\;\text{if}\;\;
\boldsymbol{\nu_2}(j)=\boldsymbol{\nu_2}(j')\right\}
\]
can also be described in terms of matrices. Since
\[
\boldsymbol{\nu_2} = {\rm vec}\begin{pmatrix}\boldsymbol{\boldsymbol{e_1}}^t+\boldsymbol{\boldsymbol{e_1}}^t & \ldots & \boldsymbol{e_d}^t+\boldsymbol{\boldsymbol{e_1}}^t\\
\vdots & & \vdots\\ \boldsymbol{\boldsymbol{e_1}}^t+\boldsymbol{e_d}^t & \ldots & \boldsymbol{e_d}^t+\boldsymbol{e_d}^t\end{pmatrix},
\]
we may write
\[
X_2\ =\ \left\{ \boldsymbol{x}\in\C^{d^2}:
\boldsymbol{x}={\rm vec}(X),\;\; X =X^t\in\C^{d\times d}\right\}.
\]
Alternatively, as in \cite[\S2.3]{VLV15},  one may permute the standard basis 
vectors $\boldsymbol{\boldsymbol{e_1}},\,\ldots,\,\boldsymbol{e_{d^2}}\in\C^{d^2}$
according to the $d^2\times d^2$ permutation matrix
\[
\Pi_{dd}\ =\
\left( \boldsymbol{e_{1+0\cdot d}},\,\boldsymbol{e_{1+1\cdot d}},\,\ldots,\,\boldsymbol{e_{1+(d-1)\cdot d}},\,
\ldots,\,\boldsymbol{e_{d+0\cdot d}},\,\boldsymbol{e_{d+1\cdot d}},\,\ldots,\,\boldsymbol{e_{d^2}} \right)
\]
and characterize the symmetric subspace as
\[
X_2\ =\ \left\{\boldsymbol{x}\in\C^{d^2}: \Pi_{dd}\,\boldsymbol{x} =\boldsymbol{x}\right\}.
\]
More generally, the higher order symmetric spaces $X_n$ can also be described 
in terms of higher order tensors. 
A tensor $X\in\mathbb C^{d\times\cdots\times d}$ of order $n$ is called symmetric, if
\[
X_{i_1,\ldots,i_n}  = X_{\sigma(i_1),\ldots,\sigma(i_n)},\qquad\text{for any permutation}\,\sigma\in S_d,
\]
see for example \cite[Section 2.2]{KB09} or \cite[Chapter 3.5]{Hack12}, and an inductive argument with respect to $n$ shows that
\[
X_n\ =\ \left\{ \boldsymbol{x}\in\C^{d^n}:
\boldsymbol{x}={\rm vec}(X),\;\; X\in\C^{d\times\cdots\times d}\ \text{symmetric}\right\}.
\]

\subsection{Relation between the subspaces}
Due to the recursive definition of the redundant multi-index enumeration,
the symmetric subspaces of neighboring order can be easily related to each other
as follows.

\begin{lemma}\label{lem:dec}
The symmetric subspace $X_{n+1}$ is contained in the
$d$-ary Cartesian product of the symmetric subspace $X_n$, 
\[
X_{n+1} \subseteq X_n \times \cdots \times X_n,\qquad n\in\N.
\]
\end{lemma}

\begin{proof}
We decompose a vector 
\[
\boldsymbol{x}\ =\ (\boldsymbol{\boldsymbol{x^{(1)}}},\,\ldots,\,\boldsymbol{\boldsymbol{x^{(d)}}})^t\in X_{n+1}
\] 
into $d$ subvectors with $d^n$ components each.
The $d^{n+1}$ components of $\boldsymbol{x}$ can be labelled by the multi-indices 
\[
\boldsymbol{\nu_n}(1)+\boldsymbol{\boldsymbol{e_1}}^t,\,\ldots,\,\boldsymbol{\nu_n}(d^n)+\boldsymbol{\boldsymbol{e_1}}^t,\,\ldots,\,
\boldsymbol{\nu_n}(1)+\boldsymbol{e_d}^t,\,\ldots,\,\boldsymbol{\nu_n}(d^n)+\boldsymbol{e_d}^t,
\] 
so that the components of the subvector $\boldsymbol{x^{(m)}}$, $m=1,\ldots,d$,
can be labelled by
\[
\boldsymbol{\nu_n}(1)+\boldsymbol{e_m}^t,\,\ldots,\,\boldsymbol{\nu_n}(d^n)+\boldsymbol{e_m}^t.
\]
Hence, 
\[
x^{(m)}_j = x^{(m)}_{j'}\quad\mbox{if}\quad \boldsymbol{\nu_n}(j) = \boldsymbol{\nu_n}(j'),
\] 
and $\boldsymbol{x^{(m)}}\in X_n$ for all $m=1,\,\ldots,\,d$.
\end{proof}

The two-dimensional examples 
\[
X_1=\C^2\quad\mbox{and}\quad
X_2=\left\{\boldsymbol{x}\in\C^4: x_2=x_3\right\}
\] 
show that the inclusion of 
Lemma~\ref{lem:dec} is in general not an equality.

\subsection{An orthonormal basis}
We now use the standard basis of $\C^{d^n}$ to construct an orthonormal 
basis of the symmetric subspace $X_n$. 

\begin{lemma} 
Let $\boldsymbol{\boldsymbol{e_1}},\ldots,\boldsymbol{e_{d^n}}$ be the standard basis vectors of $\C^{d^n}$, and 
define the vectors
\[
\boldsymbol{p_i}\ =\ \frac{1}{\sqrt{\# \sigma_n(i)}}\ \sum_{j\in\sigma_n(i)}\boldsymbol{e_j},\qquad
i=1,\,\ldots,\,L_n.
\]
Then, $\{\boldsymbol{p_1},\,\ldots,\,\boldsymbol{p_{L_n}}\}$ forms an orthonormal basis of the space 
$X_n$.
\end{lemma}

\begin{proof}
For all $i=0,\,\ldots,\,L_n$ and $j,\,j'=1,\,\ldots,\,d^n$, we have
\[
(\boldsymbol{p_i})_j\ =\ \left\{\begin{array}{ll} (\#\sigma_n(i))^{-1/2},
&\,\text{if}\ j\in\sigma_n(i),\\ 0, &\,\text{otherwise.}\end{array}\right.
\]
Since $j\in\sigma_n(i)$ if and only if $\boldsymbol{\nu_n}(j)=\boldsymbol{\ell_n}(i)$, we have 
\[
(\boldsymbol{p_i})_j\ =\ (\boldsymbol{p_i})_{j'} \quad\text{if}\quad \boldsymbol{\nu_n}(j)=\boldsymbol{\nu_n}(j'),
\]
and thus $\boldsymbol{p_i}\in X_n$. We also observe, that for all $i,\,i'=1,\,\ldots,\,L_n$,
\begin{align*}
\langle \boldsymbol{p_i},\,\boldsymbol{p_{i'}}\rangle\ &=\
\frac{1}{\sqrt{\#\sigma_n(i)\cdot\#\sigma_n(i')}}\,
\sum_{j\in\sigma_n(i)}\,\sum_{j'\in\sigma_n(i')}\,
\langle \boldsymbol{e_j},\,\boldsymbol{e_{j'}}\rangle\ \\
&=\ \delta_{i,i'}.
\end{align*}
Hence, the vectors $\boldsymbol{p_1},\,\ldots,\,\boldsymbol{p_{L_n}}$ are orthonormal.
Moreover, for all $\boldsymbol{x}\in X_n$, we have
\begin{align*}
\langle \boldsymbol{p_i},\,\boldsymbol{x}\rangle\ &=\ \frac{1}{\sqrt{\#\sigma_n(i)}}\,
\sum_{j\in\sigma_n(i)}\,\langle \boldsymbol{e_j},\,\,\boldsymbol{x}\rangle\\
&=\
\sqrt{\#\sigma_n(i)}\,x_{\boldsymbol{\ell_n}(i)},
\end{align*}
and therefore
\begin{align*}
\boldsymbol{x}\ &=\ \sum_{j=1}^{d^n}\,\langle \boldsymbol{e_j},\,\boldsymbol{x}\rangle\,\boldsymbol{e_j}\ =\ 
\sum_{i=1}^{L_n}\,\sum_{j\in\sigma_n(i)}\,
\langle \boldsymbol{e_j},\,\boldsymbol{x}\rangle\,\boldsymbol{e_j}\\
&=\ 
\sum_{i=1}^{L_n}\,x_{\boldsymbol{\ell_n}(i)}\,\sqrt{\#\sigma_n(i)}\,\boldsymbol{p_i}
 \ =\ \sum_{i=1}^{L_n}\,\langle \boldsymbol{p_i},\,\boldsymbol{x}\rangle\,\boldsymbol{p_i}.
\end{align*}
\end{proof}

The orthonormal basis $\{\boldsymbol{p_1},\ldots,\boldsymbol{p_{L_n}}\}$ of the symmetric subspace 
$X_n$ may be viewed as a normalized version of an orthogonal basis 
\[
\big\{X^{(1)},\ldots,X^{(L_n)}\big\}
\] 
of the space of $d$-dimensional symmetric tensors of order $n$ constructed as follows: For each $i=1,\ldots,L_n$ the multi-index $\boldsymbol{\ell_n}(i)=(k_1,\ldots,k_d)$ defines the non-zero elements of the corresponding basis tensor $X^{(i)}\in\C^{d\times\cdots\times d}$ according to 
\[
X^{(i)}_{j_1,\ldots,j_d}\neq 0 \quad\mbox{if}\quad (j_1,\ldots,j_d) 
\mbox{ comprises $k_1$ times $1$, \ldots, $k_d$ times $d$.}
\]
Moreover, by symmetry, all non-vanishing entries of the tensor $X^{(i)}$ have to be the same. 
The following Table~\ref{tab:basis} illustrates this alternative line of thought for the third order symmetric subspace in dimension $d=2$.
\begin{table}[h]
\begin{tabular}{c|c|c|c|c}
$i$ & $\boldsymbol{\ell_3}(i)$ & $\boldsymbol{p_i}$ & non-zero elements of $X^{(i)}$ & $\#\sigma_3(i)$\\ \hline
$1$ & $(3,0)$ & $ \boldsymbol{e_1}$ & $(1,1,1)$ & $1$\\ 
$2$ & $(2,1)$ & $ \frac{1}{\sqrt{3}}(\boldsymbol{e_2}+\boldsymbol{e_3}+\boldsymbol{e_5})$ & 
$(2,1,1)$, $(1,2,1)$, $(1,1,2)$ & $3$\\ 
$3$ & $(1,2)$ & $ \frac{1}{\sqrt3}(\boldsymbol{e_4}+\boldsymbol{e_6}+\boldsymbol{e_7})$ & $(2,2,1)$, $(2,1,2)$, $(1,2,2)$ & $3$\\ 
$4$ & $(0,3)$ & $ \boldsymbol{e_8}$ & $(2,2,2)$ & $1$\\
\end{tabular}
\bigskip
\caption{\label{tab:basis}The table lists the orthonormal basis $\{\boldsymbol{p_1},\ldots,\boldsymbol{p_4}\}$ of the third symmetric subspace $X_3$ for dimension $d=2$. It also specifies the non-vanishing entries of a corresponding basis $\{X^{(1)},\ldots, X^{(4)}\}$ of the space of symmetric tensors of size $2\times 2\times 2$.}
\end{table}

\subsection{An orthonormal matrix}
The orthonormal basis vectors $\boldsymbol{p_1},\ldots,\boldsymbol{p_{L_n}}\in X_n$
allow us to define the sparse rectangular $L_n\times d^n$ matrix
\[
P_n\ =\ \begin{pmatrix}\boldsymbol{p_1}^t\\ \vdots\\ \boldsymbol{p_{L_n}}^t\end{pmatrix}
\]
that has the $L_n$ basis vectors as its rows. For example, for $d=2$, we have
\begin{align*}
P_2\ &=\ \begin{pmatrix}\boldsymbol{\boldsymbol{e_1}}^t\\ \frac{1}{\sqrt{2}}(\boldsymbol{e_2}^t+\boldsymbol{e_3}^t)\\ 
\boldsymbol{e_4}^t\end{pmatrix}\ \\
&=\
\begin{pmatrix}1&0&0&0\\ 0&\frac{1}{\sqrt2}&\frac{1}{\sqrt2}&0\\
0&0&0&1\end{pmatrix},\\*[2ex]
P_3\ &=\ \begin{pmatrix}\boldsymbol{\boldsymbol{e_1}}^t\\
\frac{1}{\sqrt{3}}(\boldsymbol{e_2}^t+\boldsymbol{e_3}^t+\boldsymbol{e_5}^t)\\
\frac{1}{\sqrt3}(\boldsymbol{e_4}^t+\boldsymbol{e_6}^t+\boldsymbol{e_7}^t)\\ \boldsymbol{e_8}^t\end{pmatrix} 
\ \\
&=\ \begin{pmatrix}1&0&0&0&0&0&0&0\\
0&\frac{1}{\sqrt3}&\frac{1}{\sqrt3}&0&\frac{1}{\sqrt3}&0&0&0\\ 
0&0&0&\frac{1}{\sqrt3}&0&\frac{1}{\sqrt3}&\frac{1}{\sqrt3}&0\\
0&0&0&0&0&0&0&1\end{pmatrix}.
\end{align*}

\bigskip
We summarize some properties of the matrix $P_n$ and
of its adjoint $P_n^*$ and calculate explicit formulas for their actions on
vectors.

\begin{proposition}\label{prop:pn}
The $L_n\times d^n$ matrix $P_n$ and its adjoint $P_n^*$ satisfy
\[
P_n\,P_n^*\ =\ \Id_{L_n\times L_n},\qquad {\rm range}(P_n^*) \ =\ X_n.
\]
Moreover, for all $\boldsymbol{x}\in X_n$, 
\[
(P_n\,\boldsymbol{x})_i \ =\
\sqrt{\#\sigma_n(i)}\ x_{\boldsymbol{\ell_n}(i)},\qquad i=1,\,\ldots,\,L_n,
%\sqrt{\binom{n}{\boldsymbol{\ell_n}(i)}}\ x_{\boldsymbol{\ell_n}(i)},\qquad i=1,\,\ldots,\,L_n,
\]
and for all $\boldsymbol{y}\in\C^{L_n}$, 
\[
%(P_n^*\,y)_{\boldsymbol{\ell_n}(i)}\ =\ 1\left/ \sqrt{\binom{n}{\boldsymbol{\ell_n}(i)}}\right.\ y_i,
(P_n^*\,\boldsymbol{y})_{\boldsymbol{\ell_n}(i)}\ =\ 1\left/ \sqrt{\#\sigma_n(i)}\right.\ y_i,
\qquad i=1,\,\ldots,\,L_n.
\]
In particular, 
\[
P_n^*\,P_n\,\boldsymbol{x} \ =\  \boldsymbol{x},\quad\text{whenever}\;\; \boldsymbol{x}\in X_n.
\]
\end{proposition}

\begin{proof} 
The two properties $P_n\,P_n^*\ =\ \Id_{L_n\times L_n}$ and
${\rm range}(P_n^*) \ =\ X_n$  equivalently say, that 
the row vectors of $P_n$ build an orthonormal basis of $X_n$. 

\medskip
For any $y\in\C^{L_n}$, the vector $P_n^*\,y$ is a linear combination
of the column vectors $\boldsymbol{p_1},\,\ldots,\,\boldsymbol{p_{L_n}}$ and therefore in $X_n$. 
Labelling its components by $\boldsymbol{\ell_n}(1),\,\ldots,\,\boldsymbol{\ell_n}(L_n)$, we obtain
\begin{align*}
(P_n^*\,\boldsymbol{y})_{\boldsymbol{\ell_n}(i)}\ &=\ \sum_{i'=1}^{L_n}\,(p_{i'})_{\boldsymbol{\ell_n}(i)}\,y_{i'}
\\ 
&=\ \frac{1}{\sqrt{\#\sigma_n(i)}}\ y_i,\qquad i=1,\,\ldots,\,L_n.
\end{align*}
For $\boldsymbol{x}\in X_n$ and $i=1,\,\ldots,\,L_n$, we obtain
\begin{align*}
(P_n\,\boldsymbol{x})_i\ &=\ \sum_{j=1}^{d^n} (\boldsymbol{p_i})_{\boldsymbol{\nu_n}(j)}\,
x_{\boldsymbol{\nu_n}(j)} \\
&= \frac{\#\sigma_n(i)}{\sqrt{\#\sigma_n(i)}}\
x_{\boldsymbol{\ell_n}(i)} =\ \sqrt{\#\sigma_n(i)}\ x_{\boldsymbol{\ell_n}(i)},
\end{align*}
since  $\#\sigma_n(i)$ components of $\boldsymbol{p_i}$ do not vanish. In particular, 
\[
(P_n^*\,P_n\,\boldsymbol{x})_{\boldsymbol{\ell_n}(i)}\ =\ \frac{1}{\sqrt{\#\sigma_n(i)}}\
(P_n \boldsymbol{x})_i\ =\ x_{\boldsymbol{\ell_n}(i)}.
\]
\end{proof}

\section{Symmetric Kronecker products}\label{sec:kron}

\subsection{Iterated Kronecker products}
We next investigate the action of an $n$-fold Kronecker product on the
symmetric subspace $X_n$, $n\in\N$. First, we prove the invariance 
of the symmetric spaces under multiplication with iterated Kronecker products.
\begin{lemma}\label{lem:inv} 
For all $M\in\C^{d\times d}$ we have $M^{n\otimes}\boldsymbol{x} \in X_n$ whenever $\boldsymbol{x}\in X_n$.
\end{lemma}
\begin{proof} Applying $M^{n\otimes}$ to a tensor $X\in\C^{d\times\cdots\times d}$ of order $n$, 
we obtain the tensor
\[
(M^{n\otimes}X)_{i_1,\ldots,i_n} = \sum_{j_1=1}^d \cdots \sum_{j_n=1}^d M_{i_1,j_1}\cdots M_{i_n,j_n} X_{j_1,\ldots,j_n}.
\]
For a symmetric tensor $X$, we then have
\begin{align*}
(M^{n\otimes}X)_{\sigma(i_1),\ldots,\sigma(i_n)} 
&= \sum_{j_1=1}^d \cdots \sum_{j_n=1}^d M_{\sigma(i_1),j_1}\cdots M_{\sigma(i_n),j_n} X_{j_1,\ldots,j_n}\\
&= \sum_{k_1=1}^d \cdots \sum_{k_n=1}^d M_{i_1,k_1}\cdots M_{i_n,k_n} X_{k_1,\ldots,k_n}\\
&=(M^{n\otimes}X)_{i_1,\ldots,i_n}.
\end{align*}
That is, $M^{n\otimes}X$ is a symmetric tensor, too. By vectorisation we then obtain that $M^{n\otimes}\boldsymbol{x}\in X_n$ whenever $\boldsymbol{x}\in X_n$.
\end{proof}
We now derive an explicit formula for the components of a vector $M^{n\otimes} \boldsymbol{x}$  in terms of the row vectors of the matrix $M$. 

\begin{proposition}\label{prop:kron} 
Let $M\in\C^{d\times d}$, and denote by $\boldsymbol{m_1},\,\ldots,\,\boldsymbol{m_d}\in\C^d$
the row vectors of the matrix $M$. Then, for all $\boldsymbol{x}\in X_n$, the components of the vector $M^{n\otimes}\boldsymbol{x}\in X_n$ 
can be labelled by multi-indices $\boldsymbol{k}\in\N^d$ with $|\boldsymbol{k}|=n$ and satisfy
\[
\left(M^{n\otimes}\boldsymbol{x}\right)_k\ =\ \sum_{|\boldsymbol{\alpha_1}|=k_1}\,\cdots\,
\sum_{|\boldsymbol{\alpha_d}|=k_d}\,\binom{k_1}{\boldsymbol{\alpha_1}}\,\cdots\,
\binom{k_d}{\boldsymbol{\alpha_d}}\,\boldsymbol{m_1}^{\boldsymbol{\alpha_1}}\,\cdots\,\boldsymbol{m_d}^{\boldsymbol{\alpha_d}}\
x_{\boldsymbol{\alpha_1}+\cdots+\boldsymbol{\alpha_d}},
\]  
where the summation ranges over $\boldsymbol{\alpha_1},\,\ldots,\,\boldsymbol{\alpha_d}\in\N^d$ with
$|\boldsymbol{\alpha_1}|=k_1,\ldots,|\boldsymbol{\alpha_d}| = k_d$.
\end{proposition}

\begin{proof}
For $n=1$, we have $M^{n\otimes} = M$ and $X_n=\C^d$,
and our formula reduces to usual matrix-vector multiplication written as 
\[
(M\boldsymbol{x})_{e_k}\ =\ \sum_{j=1}^d\,\boldsymbol{\boldsymbol{m_j}}^{\boldsymbol{e_k}}\,x_{\boldsymbol{e_j}},\qquad
k=1,\,\ldots,\,d.
\]
For the inductive step, we consider
$\boldsymbol{x}=(\boldsymbol{x^{(1)}},\,\ldots,\,\boldsymbol{x^{(d)}})\in X_{n+1}$ decomposed as in 
Lemma~\ref{lem:dec} with $\boldsymbol{x^{(j)}}\in X_n$. We compute
\begin{align*}
M^{(n+1)\otimes}\boldsymbol{x}\ &=\ \begin{pmatrix}
m_{11}\,M^{n\otimes}&\ldots &m_{1d}\,M^{n\otimes}\\ 
\vdots &&\vdots\\ m_{d1}\,M^{n\otimes}&\ldots&m_{dd}\,M^{n\otimes}
\end{pmatrix}\begin{pmatrix}\boldsymbol{x^{(1)}}\\ \vdots \\\boldsymbol{x^{(d)}}
\end{pmatrix}
\\*[2ex] 
&=\
\begin{pmatrix}
m_{11}\,M^{n\otimes} \boldsymbol{x^{(1)}} +\cdots +m_{1d}\,M^{n\otimes}\boldsymbol{x^{(d)}}\\ 
\vdots \\
m_{d1}\,M^{n\otimes} \boldsymbol{x^{(1)}}+\cdots +m_{dd}\,M^{n\otimes} \boldsymbol{x^{(d)}}
\end{pmatrix}.
\end{align*}
By Lemma~\ref{lem:inv} we have for all $j=1,\,\ldots,\,d$, that
\[
m_{j1}\,M^{n\otimes} \boldsymbol{x^{(1)}}+\cdots+m_{jd}\,M^{n\otimes}\boldsymbol{x^{(d)}}
\in X_n.
\]
The components of these vectors can be labelled by $\boldsymbol{k}\in\N^d$ with
$|\boldsymbol{k}|=n$, and we have
\begin{align*}
&\left(m_{j1}M^{n\otimes}\boldsymbol{x^{(1)}}+\cdots+
m_{jd}\,M^{n\otimes}\boldsymbol{x^{(d)}}\right)_k\\
&=\
\sum_{|\boldsymbol{\alpha_1}|=k_1}\,\cdots\,\sum_{|\boldsymbol{\alpha_d}|=k_d}\,
\binom{k_1}{\boldsymbol{\alpha_1}}\,\cdots\,\binom{k_d}{\boldsymbol{\alpha_d}}\, 
\left(m_{j1}\, \boldsymbol{m_1}^{\boldsymbol{\alpha_1}}\cdots \boldsymbol{m_d}^{\boldsymbol{\alpha_d}}\,
x^{(1)}_{\boldsymbol{\alpha_1}+\cdots+\boldsymbol{\alpha_d}}+\,\cdots\right.\\
&\hspace*{21em} \left.+\ 
m_{jd}\,\boldsymbol{m_1}^{\boldsymbol{\alpha_1}}\cdots \boldsymbol{m_d}^{\boldsymbol{\alpha_d}}
\,x^{(d)}_{\boldsymbol{\alpha_1}+\cdots+\boldsymbol{\alpha_d}}\right).
\end{align*}
The $j$th of these sums can be rewritten as
\begin{align*}
& \sum_{|\boldsymbol{\alpha_j}|=k_j}\,\binom{k_j}{\boldsymbol{\alpha_j}}\, 
\left(\boldsymbol{m_j}^{\boldsymbol{\alpha_j}+\boldsymbol{\boldsymbol{e_1}}}\,
x^{(1)}_{\boldsymbol{\alpha_1}+\cdots+\boldsymbol{\alpha_d}}+\,
\cdots\,+\boldsymbol{m_j}^{\boldsymbol{\alpha_j}+\boldsymbol{e_d}}\,
x^{(d)}_{\boldsymbol{\alpha_1}+\cdots+\boldsymbol{\alpha_d}}\right)
\\
&=\ \sum_{|\boldsymbol{\beta_j}-\boldsymbol{\boldsymbol{e_1}}|=k_j}\,\binom{k_j}{\boldsymbol{\beta_j}-\boldsymbol{\boldsymbol{e_1}}}\, 
\boldsymbol{m_j}^{\boldsymbol{\beta_j}}\,
x^{(1)}_{\boldsymbol{\alpha_1}+\cdots+(\boldsymbol{\beta_j}-\boldsymbol{\boldsymbol{e_1}})+\cdots+\boldsymbol{\alpha_d}}+\,\cdots\,+\\
&\hspace{12em}
\sum_{|\boldsymbol{\beta_j}-\boldsymbol{e_d}|=k_j}\,\binom{k_j}{\boldsymbol{\beta_j}-\boldsymbol{e_d}}\,
\boldsymbol{m_j}^{\boldsymbol{\beta_j}}\,
x^{(d)}_{\boldsymbol{\alpha_1}+\cdots+(\boldsymbol{\beta_j}-\boldsymbol{e_d})+\cdots+\boldsymbol{\alpha_d}}.
\end{align*}
Now we observe that for all $r=1,\ldots,d$,
\[
x^{(r)}_{\boldsymbol{\alpha_1}+ \cdots+(\boldsymbol{\beta_j}-\boldsymbol{e_r})+\cdots+\boldsymbol{\alpha_d}} = x_{\boldsymbol{\alpha_1}+\cdots+\boldsymbol{\beta_j}+\cdots+\boldsymbol{\alpha_d}},
\]
so that
\begin{align*}
& \sum_{|\boldsymbol{\alpha_j}|=k_j}\,\binom{k_j}{\boldsymbol{\alpha_j}}\, 
\left(\boldsymbol{m_j}^{\boldsymbol{\alpha_j}+\boldsymbol{\boldsymbol{e_1}}}\,
x^{(1)}_{\boldsymbol{\alpha_1}+\cdots+\boldsymbol{\alpha_d}}+\,
\cdots\,+\boldsymbol{m_j}^{\boldsymbol{\alpha_j}+\boldsymbol{e_d}}\,
x^{(d)}_{\boldsymbol{\alpha_1}+\cdots+\boldsymbol{\alpha_d}}\right)
\\
&=\ \sum_{|\boldsymbol{\beta_j}-\boldsymbol{\boldsymbol{e_1}}|=k_j}\,\binom{k_j}{\boldsymbol{\beta_j}-\boldsymbol{\boldsymbol{e_1}}}\, 
\boldsymbol{m_j}^{\boldsymbol{\beta_j}}\,
x_{\boldsymbol{\alpha_1}+\cdots+\boldsymbol{\beta_j}+\cdots+\boldsymbol{\alpha_d}}+\,\cdots\,+\\
&\hspace{12em}
\sum_{|\boldsymbol{\beta_j}-\boldsymbol{e_d}|=k_j}\,\binom{k_j}{\boldsymbol{\beta_j}-\boldsymbol{e_d}}\,
\boldsymbol{m_j}^{\boldsymbol{\beta_j}}\,
x_{\boldsymbol{\alpha_1}+\cdots+\boldsymbol{\beta_j}+\cdots+\boldsymbol{\alpha_d}}.
\end{align*}
Since all multi-indices $\boldsymbol{\beta_j}\in\N^d$ with $|\boldsymbol{\beta_j}|=k_j+1$ satisfy
\[
\binom{k_j+1}{\boldsymbol{\beta_j}} = \binom{k_j}{\boldsymbol{\beta_j}-\boldsymbol{\boldsymbol{e_1}}}+ \cdots + \binom{k_j}{\boldsymbol{\beta_j}-\boldsymbol{e_d}},
\]
we can write
\begin{align*}
& \sum_{|\boldsymbol{\alpha_j}|=k_j}\,\binom{k_j}{\boldsymbol{\alpha_j}}\, 
\left(\boldsymbol{m_j}^{\boldsymbol{\alpha_j}+\boldsymbol{\boldsymbol{e_1}}}\,
x^{(1)}_{\boldsymbol{\alpha_1}+\cdots+\boldsymbol{\alpha_d}}+\,
\cdots\,+\boldsymbol{m_j}^{\boldsymbol{\alpha_j}+\boldsymbol{e_d}}\,
x^{(d)}_{\boldsymbol{\alpha_1}+\cdots+\boldsymbol{\alpha_d}}\right)
\\
&=\ \sum_{|\boldsymbol{\beta_j}|=k_j+1}\,\binom{k_j+1}{\boldsymbol{\beta_j}}\, 
\boldsymbol{m_j}^{\boldsymbol{\beta_j}}\,x_{\boldsymbol{\alpha_1}+\cdots+\boldsymbol{\beta_j}+\cdots+ \boldsymbol{\alpha_d}}
\end{align*}
and obtain
\begin{align*}
&\left(m_{j1} M^{n\otimes}\boldsymbol{x^{(1)}}+\,\cdots\,+ m_{jd}\, 
M^{n\otimes}\boldsymbol{x^{(d)}}\right)_k\ =\\
&\sum_{|\boldsymbol{\alpha_1}|=k_1}\cdots\sum_{|\boldsymbol{\alpha_j}|=k_j+1}\cdots 
\sum_{|\boldsymbol{\alpha_d}|=k_d}\,\binom{k_1}{\boldsymbol{\alpha_1}}\cdots
\binom{k_j+1}{\boldsymbol{\alpha_j}}\,\cdots\,\binom{k_d}{\boldsymbol{\alpha_d}}\\
& \hspace*{20em}
\boldsymbol{m_1}^{\boldsymbol{\alpha_1}}\cdots \boldsymbol{m_d}^{\boldsymbol{\alpha_d}} x_{\boldsymbol{\alpha_1}+\cdots+\boldsymbol{\alpha_d}}.
\end{align*}
Hence, $M^{(n+1)\otimes}\boldsymbol{x}$ has at most $L_{n+1}$ 
distinct components that can be labelled by $\boldsymbol{k}\in\N^d$ with $|\boldsymbol{k}|=n+1$.
%Altogether, we have proven 
%\[
%M^{(n+1)\otimes}\boldsymbol{x}\in X_{n+1},\qquad \mbox{whenever }\boldsymbol{x}\in X_{n+1},
%\] 
%and for all $\boldsymbol{k}\in\N^d$ with $|\boldsymbol{k}|=n+1$,
They satisfy
\[
\left(M^{(n+1)\otimes}\boldsymbol{x}\right)_k\ =\
\sum_{|\boldsymbol{\alpha_1}|=k_1}\,\cdots\,\sum_{|\boldsymbol{\alpha_d}|=k_d}\,
\binom{k_1}{\boldsymbol{\alpha_1}}\,\cdots\,\binom{k_d}{\boldsymbol{\alpha_d}}\, 
\boldsymbol{m_1}^{\boldsymbol{\alpha_1}}\,\cdots\,\boldsymbol{m_d}^{\boldsymbol{\alpha_d}}\ x_{\boldsymbol{\alpha_1}+\cdots+\boldsymbol{\alpha_d}}.
\]
\end{proof}

\begin{remark}
The invariance property of Lemma~\ref{lem:inv} can also be proved alongside 
the inductive argument given in the proof of Proposition~\ref{prop:kron} without 
using tensor terminology.
\end{remark}

\subsection{Symmetric Kronecker products}
Having proven that $n$-fold Kronecker products leave the $n$th symmetric 
subspace invariant, we define the $n$-fold symmetric Kronecker product as 
follows:

\begin{definition}
For $M\in\C^{d\times d}$ and $n\in\N$,
we define the $L_n\times L_n$ matrix
\[
S_n(M) \ =\ P_n \,M^{n\otimes}\, P_n^*
\]
and call it the {\em $n$-fold symmetric Kronecker product} of the matrix $M$.
\end{definition}

The $n$-fold symmetric Kronecker product has useful structural properties.

\begin{lemma}\label{lem:str}
The $n$-fold symmetric Kronecker product $S_n(M)$ of a matrix 
$M\in\C^{d\times d}$ satisfies
$S_n(M)^* = S_n(M^*)$.
If $M\in{\rm GL}(d,\C)$, then
\[
S_n(M)\in{\rm GL}(L_n,\C)\quad\text{with}\quad S_n(M)^{-1}\ =\
S_n(M^{-1}).
\]
In particular, if $M\in U(d)$, then $S_n(M)\in U(L_d)$.
\end{lemma}

\begin{proof}
Since $(M\otimes M)^*=M^*\otimes M^*$ and
$(M^{n\otimes})^*=(M^*)^{n\otimes}$, we have 
\begin{align*}
S_n(M^*)\ &=\ P_n(M^{n\otimes})^*P_n^*\\
& =\ S_n(M^*).
\end{align*}
If $M$ is invertible, then $M\otimes M$ is invertible with
$(M\otimes M)^{-1} = M^{-1}\otimes M^{-1}$.
By Proposition~\ref{prop:kron},
\[
(M^{n\otimes})^{-1}\boldsymbol{p_j}\in X_n,\qquad j=1,\,\ldots,\,L_n.
\]
Proposition~\ref{prop:pn} yields $P_n^*\,P_n\,\boldsymbol{x}=\boldsymbol{x}$ for all $\boldsymbol{x}\in X_n$ and 
\[
P_n\,\boldsymbol{p_j} \ =\ \boldsymbol{e_j},\qquad j=1,\,\ldots,\,L_n,
\] 
where $\boldsymbol{\boldsymbol{e_1}},\,\ldots,\,\boldsymbol{e_{L_n}}\in\C^{L_n}$ are the standard basis vectors. 
Therefore,
\begin{align*}
S_n(M)\,S_n(M^{-1})\,\boldsymbol{e_j} \ &=\
P_n\,M^{n\otimes}\,P_n^*\,P_n\,(M^{n\otimes})^{-1}\boldsymbol{p_j}\\
& =\
P_n\,\boldsymbol{p_j}\\
& =\ \boldsymbol{e_j},
\end{align*}
that is, 
\[
S_n(M)\,S_n(M^{-1})\ = \ {\rm Id}_{L_n\times L_n}.
\]
\end{proof}

\subsection{The main result}

The explicit formula of Proposition~\ref{prop:kron}
for the $n$-fold Kronecker product also 
allows a detailed description of the $n$-fold symmetric Kronecker product.

\begin{theorem}\label{theo:main}
Let $M\in\C^{d\times d}$, and denote by $\boldsymbol{m_1},\,\ldots,\,\boldsymbol{m_d}\in\C^d$
the row vectors of the matrix $M$. Then, the $n$-fold symmetric Kronecker 
product satisfies for all $\boldsymbol{y}\in\C^{L_n}$ and $\boldsymbol{k}\in\N^d$ with $|\boldsymbol{k}|=n$,
\begin{align*}
&\left(S_n(M)\,\boldsymbol{y}\right)_k
=\,\frac{1}{\sqrt{\boldsymbol{k}!}}\sum_{|\boldsymbol{\alpha_1}|=k_1}\cdots
\sum_{|\boldsymbol{\alpha_d}|=k_d}\,\binom{k_1}{\boldsymbol{\alpha_1}}\cdots
\binom{k_d}{\boldsymbol{\alpha_d}}\,\boldsymbol{m_1}^{\boldsymbol{\alpha_1}}\cdots \boldsymbol{m_d}^{\boldsymbol{\alpha_d}}\\
&\hspace*{20em}
\times \ \sqrt{(\boldsymbol{\alpha_1}+\cdots+\boldsymbol{\alpha_d})!}\ y_{\boldsymbol{\alpha_1}+\cdots+\boldsymbol{\alpha_d}},
\end{align*}
where the summations range over $\boldsymbol{\alpha_1},\ldots,\boldsymbol{\alpha_d}\in\N^d$ with 
$|\boldsymbol{\alpha_1}|=k_1$, \ldots, $|\boldsymbol{\alpha_d}| = k_d$, and the components of $\boldsymbol{y}\in\C^{L_n}$ are denoted by multi-indices of order $n$.
\end{theorem}

\begin{proof} 
By Lemma~\ref{lem:sigma}, we have
\[
\#\sigma_n(i) = \binom{n}{\boldsymbol{\ell_n}(i)},\qquad\mbox{for }i=1,\,\ldots,\,L_n.
\]
By Proposition~\ref{prop:pn}, we have $P_n^*\,\boldsymbol{y}\in X_n$ and
\[
(P_n^*\,\boldsymbol{y})_{\boldsymbol{\ell_n}(i)}\ =\ 1\left/\sqrt{\binom{n}{\boldsymbol{\ell_n}(i)}}\ \right.y_i,
\qquad\mbox{for }i=1,\,\ldots,\,L_n.
\]
This can be reformulated as
\[
(P_n^*\,\boldsymbol{y})_\alpha = \sqrt{\frac{\alpha !}{n!}} \ y_{\boldsymbol{\alpha}},\qquad 
\mbox{for }\boldsymbol{\alpha}\in\N^d\ \mbox{with }|\boldsymbol{\alpha}| = n.
\]
By Lemma~\ref{lem:inv}, the $n$-fold Kronecker product leaves
$X_n$ invariant, and by Proposition~\ref{prop:kron} we have for all $\boldsymbol{k}\in\N^d$ with $|\boldsymbol{k}|=n$,
\begin{align*}
&\left(M^{n\otimes}\,P_n^*\,\boldsymbol{y}\right)_k\ \\
&=\
\sum_{|\boldsymbol{\alpha_1}|=k_1}\,\cdots\,\sum_{|\boldsymbol{\alpha_d}|=k_d} 
\binom{k_1}{\boldsymbol{\alpha_1}}\,\cdots\,\binom{k_d}{\boldsymbol{\alpha_d}}\,
\boldsymbol{m_1}^{\boldsymbol{\alpha_1}}\,\cdots\,\boldsymbol{m_d}^{\boldsymbol{\alpha_d}}\, 
(P_n^*\,\boldsymbol{y})_{\boldsymbol{\alpha_1}+\cdots+\boldsymbol{\alpha_d}}\\
&=\ \sum_{|\boldsymbol{\alpha_1}|=k_1}\,\cdots\,\sum_{|\boldsymbol{\alpha_d}|=k_d}\, 
\binom{k_1}{\boldsymbol{\alpha_1}}\,\cdots\,\binom{k_d}{\boldsymbol{\alpha_d}}\,
\boldsymbol{m_1}^{\boldsymbol{\alpha_1}}\,\cdots\,\boldsymbol{m_d}^{\boldsymbol{\alpha_d}}\\
& \hspace*{20em}
\times\ \sqrt{\frac{(\boldsymbol{\alpha_1}+\cdots+\boldsymbol{\alpha_d})!}{n!}}\
y_{\boldsymbol{\alpha_1}+\cdots+\boldsymbol{\alpha_d}}.
\end{align*}
By Proposition~\ref{prop:pn}, we have for all $\boldsymbol{x}\in X_n$
\[
(P_n\,\boldsymbol{x})_i\ =\ \sqrt{\binom{n}{\boldsymbol{\ell_n}(i)}}\ x_{\boldsymbol{\ell_n}(i)},
\qquad\mbox{for } i=1,\,\ldots,\,L_n,
\] 
so that
\begin{align*}
&\left(P_n\,M^{n\otimes}\,P_n^*\,\boldsymbol{y}\right)_k\\
&=\ \sqrt{\frac{n!}{\boldsymbol{k}!}}\sum_{|\boldsymbol{\alpha_1}|=k_1}\cdots
\sum_{|\boldsymbol{\alpha_d}|=k_d} 
\binom{k_1}{\boldsymbol{\alpha_1}}\cdots\binom{k_d}{\boldsymbol{\alpha_d}}
\boldsymbol{m_1}^{\boldsymbol{\alpha_1}}\cdots \boldsymbol{m_d}^{\boldsymbol{\alpha_d}}\\
&\hspace*{20em}
\times\ \sqrt{\frac{(\boldsymbol{\alpha_1}+\cdots+\boldsymbol{\alpha_d})!}{n!}}\,
y_{\boldsymbol{\alpha_1}+\cdots+\boldsymbol{\alpha_d}}\\
&=\ \frac{1}{\sqrt{\boldsymbol{k}!}}\sum_{|\boldsymbol{\alpha_1}|=k_1}\cdots
\sum_{|\boldsymbol{\alpha_d}|=k_d}\,\binom{k_1}{\boldsymbol{\alpha_1}}\cdots
\binom{k_d}{\boldsymbol{\alpha_d}}
\boldsymbol{m_1}^{\boldsymbol{\alpha_1}}\cdots \boldsymbol{m_d}^{\boldsymbol{\alpha_d}}\\
& \hspace*{20em}
\times\ \sqrt{(\boldsymbol{\alpha_1}+\cdots+\boldsymbol{\alpha_d})!}\ y_{\boldsymbol{\alpha_1}+\cdots+\boldsymbol{\alpha_d}}.
\end{align*}
\end{proof}

\section{Application to semiclassical wavepackets}\label{sec:sem}
\subsection{Parametrizing Gaussians}

We consider two complex invertible matrices $A,\,B\in{\rm GL}(d,\,\C)$
that satisfy the conditions
\begin{align}\label{eq:mat1}
A^tB - B^t A\ &=\ 0,
\\ \label{eq:mat2}
A^*B + B^* A\ &=\ 2\,{\rm Id}_{d\times d}.
\end{align}
These two conditions imply that $B\,A^{-1}$
is a complex symmetric matrix such that its real part %satisfies
\[
{\rm Re}(B\,A^{-1})\ =\ (A\,A^*)^{-1}
\]
is a Hermitian and positive definite matrix, see \cite{Hag80}.
Let $\hbar>0$ and define
the multivariate complex-valued Gaussian function
\begin{align}\nonumber
&\varphi_{\boldsymbol{0}}[A,\,B]: \R^d\to\C,\\*[-1ex] \label{eq:gauss} \\*[-3ex] \nonumber
&\varphi_{\boldsymbol{0}}[A,\,B](\boldsymbol{x})\ =\ (\pi\,\hbar)^{-d/4}\,\det(A)^{-1/2}\,
%\exp\!\left(-\tfrac{1}{2\,\hbar}\langle x,\,B\,A^{-1}\,x\rangle\right).
\exp\!\left(\,-\,\frac{\langle \boldsymbol{x},\,B\,A^{-1}\,\boldsymbol{x}\rangle}{2\,\hbar}\right).
\end{align}
Then, $\varphi_{\boldsymbol{0}}[A,\,B]$ is a square-integrable function,
and the constant factor $\det(A)^{-1/2}$
ensures normalization according to
\[
\int_{\R^d}\,\left|\varphi_{\boldsymbol{0}}[A,\,B](\boldsymbol{x})\right|^2 \,d\boldsymbol{x}\ =\ 1.
\]
Changing the parametrization by a unitary matrix changes the Gaussian 
function only by constant multiplicative factor of modulus one:

\begin{lemma}\label{lem:gauss}
Let $A,\,B\in{\rm GL}(d,\,\C)$ satisfy the conditions 
(\ref{eq:mat1}--\ref{eq:mat2}).
Then, for all unitary matrices $U\in U(d)$,
the matrices $A'=A\,U$ and $B'=B\,U$ also satisfy the 
conditions (\ref{eq:mat1}--\ref{eq:mat2}). Moreover, 
\[
\varphi_{\boldsymbol{0}}[A',\,B'] = \det(U)^{-1/2} \,\varphi_{\boldsymbol{0}}[A,\,B].
\]
\end{lemma}

\begin{proof} We observe that
\begin{align*}
(A')^tB' -(B')^t A'\ &=\  U^t (A^t B -BA)U\ =\ 0\\
(A')^*B' + (B')^*A'\ &=\ U^*(A^* B + BA)U\ =\ 2\,{\rm Id}_{d\times d}.
\end{align*}
and $B'\,(A')^{-1} = B\,U\,U^*A^{-1} = BA^{-1}$. Therefore, 
\begin{align*}
\varphi_{\boldsymbol{0}}[A',\,B']\ &=\ (\pi\,\hbar)^{-d/4}\,\det(A\,U)^{-1/2}\,
%\exp\!\left(-\tfrac{1}{2\hbar}\langle x,\,B\,A^{-1}\,x\rangle\right)\\
\exp\left(\,-\,\frac{\langle \boldsymbol{x},\,B\,A^{-1}\,\boldsymbol{x}\rangle}{2\,\hbar}\right)\\
&=\ \det(U)^{-1/2}\,\varphi_{\boldsymbol{0}}[A,\,B].
\end{align*}
\end{proof}

\subsection{Semiclassical wave packets}
Following the construction of \cite{Hag98}, we consider
$A,\,B\in{\rm GL}(d,\,\C)$ satisfying the conditions
(\ref{eq:mat1}--\ref{eq:mat2}) and introduce the vector of raising operators
\[
\Rr[A,\,B]\ =\
\frac{1}{\sqrt{2\hbar}}\left(B^*\,\boldsymbol{x} - iA^*(-i\hbar\nabla_{\boldsymbol{x}})\right)
\]
that consists of $d$ components,  
\[
\Rr[A,\,B]\ = \
\begin{pmatrix}\Rr_1[A,\,B]\\ \vdots \\ \Rr_d[A,\,B]\end{pmatrix}.
\]
The raising operator acts on Schwartz functions $\psi:\R^d\to\C$ as
\[
\left(\Rr[A,\,B]\psi\right)(\boldsymbol{x})\ =\ \frac{1}{\sqrt{2\hbar}}
\left(B^*\,\boldsymbol{x}\,\psi(\boldsymbol{x}) - iA^*(-i\hbar\nabla_{\boldsymbol{x}}\psi)(\boldsymbol{x})\right),\qquad \boldsymbol{x}\in\R^d.
\]

Powers of the raising operator now generate the semiclassical wave packets.

\begin{definition}[Semiclassical wave packet]\label{def:packet} 
Let $A,\, B\in{\rm GL}(d,\,\C)$ satisfy the conditions (\ref{eq:mat1}--\ref{eq:mat2}) 
and $\varphi_{\boldsymbol{0}}[A,B]$ denote the Gaussian function of (\ref{eq:gauss}). Then, the $\boldsymbol{k}^{\mbox{\scriptsize th}}$
semiclassical wave packet is defined by
\[
\varphi_{\boldsymbol{k}}[A,\,B]\ =\ \frac{1}{\sqrt{\boldsymbol{k}!}}\ \Rr[A,B]^{\boldsymbol{k}}\varphi_{\boldsymbol{0}}[A,\,B],\qquad
\boldsymbol{k}\in\N^d\ .
\]
\end{definition}

We note that the $\boldsymbol{k}^{\mbox{\scriptsize th}}$ power of the raising operator
\[
\Rr[A,\,B]^{\boldsymbol{k}} = \Rr_1[A,\,B]^{k_1}\,\cdots\,\Rr_d[A,B]^{k_d}
\]
does not depend on the ordering,
since the components commute with one another due to the compatibility 
conditions (\ref{eq:mat1}--\ref{eq:mat2}).
The set 
\[
\left\{\varphi_{\boldsymbol{k}}[A,\,B]: \boldsymbol{k}\in\N^d\right\}
\] 
forms an orthonormal basis of the space of square-integrable functions 
$L^2(\R^d,\,\C)$.

\subsection{Hermite polynomials}
By its construction, the $\boldsymbol{k}^{\mbox{\scriptsize th}}$ semiclassical wave packet is a multivariate
polynomial of degree $|\boldsymbol{k}|$ in $\boldsymbol{x}$ times the complex-valued Gaussian function $\varphi_{\boldsymbol{0}}[A,\,B]$, that is, 
\[
\varphi_{\boldsymbol{k}}[A,\,B](x)\ =\ \frac{1}{\sqrt{2^{|\boldsymbol{k}|}\,\boldsymbol{k}!}}\,p_{\boldsymbol{k}}[A](\boldsymbol{x})\,\varphi_{\boldsymbol{0}}[A,\,B](\boldsymbol{x}),\qquad \boldsymbol{x}\in\R^d,
\]
The polynomials $p_{\boldsymbol{k}}[A]$ are determined by the matrix
$A\in{\rm GL}(d,\,\C)$, see \cite{Hag15}, and satisfy the three-term recurrence relation
\[
\left(p_{k+\boldsymbol{e_j}}[A]\right)_{j=1}^d\ =\ \frac{2}{\sqrt{\hbar}}
A^{-1}\boldsymbol{x}\,p_{\boldsymbol{k}}[A] - 2A^{-1}\overline{A}
\left(k_j\,p_{\boldsymbol{k}-\boldsymbol{e_j}}[A]\right)_{j=1}^d,
\]
see \cite[Chapter V.2]{L}.
Whenever the parameter matrix $A$ has all entries real,
$A\in{\rm GL}(d,\,\R)$, then the polynomials factorize according to 
\[
p_{\boldsymbol{k}}[A](\boldsymbol{x})\ =\ \prod_{j=1}^d\,
H_{k_j}\!\left(\tfrac{1}{\sqrt{\hbar}}(A^{-1}\boldsymbol{x})_j\right),\qquad \boldsymbol{x}\in\R^d,
\]
where $H_n$ is the $n^{\mbox{\scriptsize th}}$ Hermite polynomial, $n\in\N$,
defined by the univariate three-term recurrence relation
\[
H_{n+1}(y)\ =\ 2\,y\,H_n(y)\,-\,2\,n\,H_{n-1}(y),\qquad y\in\R.
\]

The real parameter case, however, is rather exceptional when using 
semiclassical wave packets for their key application in molecular quantum 
dynamics. 
There, the parameter matrices $A(t)$ and $B(t)$, $t\in\R$,
are time-dependent and determined by a system of ordinary differential 
equations. 
For the particularly simple, but instructive example of
%1--dimensional
harmonic oscillator motion,
one can even write the solution explicitly as
\begin{align*}
A(t)\ &=\ \cos(t)\,A(0) + i\,\sin(t)\,B(0),\\
B(t)\ &=\ i\,\sin(t)\,A(0) + \cos(t)\,B(0).
\end{align*}
Hence, the matrix $A(t)$
cannot be expected to have only real entries, and the crucial matrix factor 
$A(t)^{-1}\overline{A(t)}$ 
in the three term recurrence relation generates multivariate polynomials 
beyond a tensor product representation. 

\subsection{Changing the parametrization}\label{Section5.3}
If $A,\,B\in{\rm GL}(d,\,\C)$ satisfy the compatibility 
conditions (\ref{eq:mat1}--\ref{eq:mat2}),
then $|A| = \sqrt{AA^*}$ is a real symmetric, 
positive definite matrix, 
and the singular value decomposition of $A$,
\[
A\ =\ V\,\Sigma W^*\quad\text{with}\quad
\Sigma={\rm diag}(\sigma_1,\ldots,\sigma_d)\;\;\text{positive definite},
\]
is given by an orthogonal matrix $V\in O(d)$ and a unitary matrix
$W\in U(d)$. This provides two natural ways for transforming
$A' = A\,U$ with $A'\in{\rm GL}(d,\,\R)$ and $U\in U(d)$.
One may work with the polar decomposition of $A$, 
\[
A'\ =\ |A|\ =\ V\,\Sigma\,V^*\quad\text{and}\quad U=W\,V^*,
\]
or alternatively with
\[
A'\ =\ V\,\Sigma\quad\text{and}\quad U=W.
\]

Both choices provide a unitary transformation to the real case,
and we ask how to relate different families of wave packets that correspond
to unitarily linked parametrizations.
For an explicit description, we collect the semiclassical wave packets
of order $n$ in one formal vector
\[
\vec\varphi_n[A,\,B]\ =\
\begin{pmatrix}\varphi_{\boldsymbol{\ell_n}(1)}[A,\,B]\\ \vdots\\
\varphi_{\boldsymbol{\ell_n}(L_n)}[A,\,B]\end{pmatrix},
\]
whose components are labelled by the multi-indices
$\boldsymbol{\ell_n}(1),\,\ldots,\,\boldsymbol{\ell_n}(L_n)$.
Then, we use the $n$-fold symmetric Kronecker product in the following way:

\begin{corollary}\label{MainResult}
Let $A,\,B\in{\rm GL}(d,\,\C)$ satisfy the conditions
(\ref{eq:mat1}--\ref{eq:mat2}),
and consider the matrices $A'=A\,U$, $B'=B\,U$ with $U\in U(d)$.
Then,  
\[
\vec\varphi_n[A',\,B']\ =\ \det(U)^{-1/2}\,S_n(U)\,
\vec\varphi_n[A,\,B],\qquad n\in\N.
\]
\end{corollary}

\begin{proof}
We observe that the raising operators transform according to
\[
\Rr[A',\,B']\ =\ \Rr[A\,U,\,B\,U] \ =\  U^*\,\Rr[A,\,B],
\]
which means for the components  that
\[
\Rr_j[A',\,B']\ =\ \sum_{m=1}^d\,
\overline{u_{mj}}\ \Rr_m[A,\,B],\qquad j=1,\,\ldots,\,d.
\]
Since all components of the raising operators commute which each other,
we can use the multinomial theorem and obtain for all $n\in\N$ that
\[
\Rr_j[A',\,B']^n\ =\ \sum_{|\boldsymbol{\alpha}|=n}\,
\binom{n}{\boldsymbol{\alpha}}\,\overline{\boldsymbol{u_j}^{\boldsymbol{\alpha}}}\ 
\Rr[A,\,B]^{\boldsymbol{\alpha}},
\]
where $\boldsymbol{u_1},\,\ldots,\,\boldsymbol{u_d}\in\C^d$
denote the column vectors of the matrix $U$. 
This implies for any $\boldsymbol{k}\in\N^d$ with $|\boldsymbol{k}|=n$,
\[
\Rr[A',\,B']^{\boldsymbol{k}}\ =\ \sum_{|\boldsymbol{\alpha_1}| = k_1}\,\cdots\,
\sum_{|\boldsymbol{\alpha_d}|=k_d}\,
\binom{k_1}{\boldsymbol{\alpha_1}}\,\cdots\,\binom{k_d}{\boldsymbol{\alpha_d}}\, 
\overline{u_{\boldsymbol{1}}^{\boldsymbol{\alpha_1}}}\,\cdots\,\overline{\boldsymbol{u_d}^{\boldsymbol{\alpha_d}}}\
\Rr[A,\,B]^{\boldsymbol{\alpha_1}+\cdots+\boldsymbol{\alpha_d}},
\]
where the $d$ summations run over $\boldsymbol{\alpha_1},\,\ldots,\,\boldsymbol{\alpha_d}\in\N^d$.
Together with Lemma~\ref{lem:gauss}, we therefore obtain
\begin{align*}
&\varphi_{\boldsymbol{k}}[A',\,B']\ =\
\frac{1}{\sqrt{\boldsymbol{k}!}}\,\Rr[A',\,B']^{\boldsymbol{k}} \,\varphi_{\boldsymbol{0}}[A',\,B'] \\
&=\ \frac{1}{\sqrt{\det(U)\,\boldsymbol{k}!}}\,\sum_{|\boldsymbol{\alpha_1}| = k_1}\,\cdots\,
\sum_{|\boldsymbol{\alpha_d}|=k_d}\,
\binom{k_1}{\boldsymbol{\alpha_1}}\,\cdots\,\binom{k_d}{\boldsymbol{\alpha_d}}\,
\overline{u_{\boldsymbol{1}}^{\boldsymbol{\alpha_1}}}\,\cdots\,\overline{\boldsymbol{u_d}^{\boldsymbol{\alpha_d}}}\\ 
&\hspace*{21em}\times\quad\Rr[A,\,B]^{\boldsymbol{\alpha_1}+\cdots+\boldsymbol{\alpha_d}}\,
\varphi_{\boldsymbol{0}}[A,B]\\[3mm]
&=\ \frac{1}{\sqrt{\det(U)\,\boldsymbol{k}!}}\,\sum_{|\boldsymbol{\alpha_1}|=k_1}\,\cdots\,
\sum_{|\boldsymbol{\alpha_d}|=k_d}\,\binom{k_1}{\boldsymbol{\alpha_1}}\,\cdots\,
\binom{k_d}{\boldsymbol{\alpha_d}}
\;\overline{u_{\boldsymbol{1}}^{\boldsymbol{\alpha_1}}}\,\cdots\,\overline{\boldsymbol{u_d}^{\boldsymbol{\alpha_d}}}\\
&\hspace*{18em}\times\,\sqrt{(\boldsymbol{\alpha_1}+\cdots+\boldsymbol{\alpha_d})!}\
\varphi_{\boldsymbol{\alpha_1}+\cdots+\boldsymbol{\alpha_d}}[A,\,B].
\end{align*}
By Theorem~\ref{theo:main}, we then obtain
\[
\vec\varphi_n[A',\,B']\ =\ \det(U)^{-1/2}\,S_n(U)\,\vec\varphi_n[A,\,B].
\]
\end{proof}

\section{Conclusion}
We have derived an explicit formula for the action of 
$n$-fold Kronecker products on symmetric subspaces. See Theorem~\ref{theo:main}. 
Our findings generalize results on two-fold symmetric Kronecker products discussed 
in the literature on semidefinite programming \cite{AHO98}, \cite[Appendix~E]{Kle02}.
The new formula allows one to write a linear change of the parametrization of semiclassical wave packets 
in terms of a $n$-fold symmetric Kronecker product. 
Semiclassical wave packetshave an associated family of multivariate orthogonal 
polynomials. Our result provides an explicit transformation of these polynomials to a tensor 
products of scaled univariate Hermite polynomials. Moreover, semiclassical wave packets
have been used in \cite{FGL09}, \cite[Chapter~5]{L} and \cite{GH14}
for a Galerkin discretization of multi-dimensional molecular quantum dynamics. 
The explicit formula for the effect of a change of parametrization will allow to convert 
the multi-dimensional Galerkin integrals of the method to a product of one-dimensional integrals,
resulting in a more accurate and stable numerical method, see \cite[Chapter~5]{B17} for 
numerical experiments in this direction.

\subsection*{Acknowledgements} This research was partially supported by the U.S.~National Science Foundation
Grant DMS--1210982 and the German Research Foundation (DFG), Collaborative Research Center 
SFB/TRR 109. The authors thank the anonymous referees for their constructive help in improving 
the clarity of presentation.

%\appendix
%\section{The matrix $P_2$} 
%
%For $n=2$, we have $L_2 = \frac12 d(d+1)$, and the $L_2\times d^2$
%matrix $P_2$ can also be introduced via symmetric vectorisation, see 
%\cite[Appendix~E]{K02}. For a symmetric matrix $S\in\R^{d\times d}$, %one considers the columnwise vectorisation
%\[
%{\rm vec}_c(S) = \left(s_{11},s_{21},\ldots,s_{d1},s_{12},s_{22},\ldots,s_{d2},\ldots,s_{1d},s_{2d},\ldots,s_{dd}\right)^t\in\R^{d^2}
%\]
%and the symmetric vectorisation
%\[
%{\rm svec}(S) = \left(s_{11},\sqrt2 s_{21},\ldots,\sqrt2 s_{d1},s_{22},%\sqrt2 s_{32},\ldots,\sqrt2 s_{d2},\ldots,s_{dd}\right)^t\in\R^{L_2}.
%\]
%Then, 
%\[
%{\rm svec}(S) = P_2 {\rm vec}_c(S).
%\]

\end{document}